\documentclass[11pt]{amsart}
\usepackage{amsmath, amsthm, amssymb, hyperref, geometry, url}

\evensidemargin \oddsidemargin
\author{Benjamin Linowitz}
\address{Department of Mathematics\\University of Michigan\\Ann Arbor, MI 48109}
\email{linowitz@umich.edu}

\author{Jeffrey S. Meyer}
\address{Department of Mathematics\\
University of Oklahoma\\
Norman, OK 73019 USA}
\email[]{jmeyer@math.ou.edu}

\author{Paul Pollack}
\address{Department of Mathematics\\University of Georgia\\Athens, GA 30602}
\email{pollack@uga.edu}

\title{The length spectra of arithmetic hyperbolic $3$-manifolds and their totally geodesic surfaces}

\usepackage{fancyhdr}

\fancypagestyle{plain}

\DeclareMathAlphabet{\curly}{U}{rsfs}{m}{n}

\DeclareMathOperator{\Ann}{Ann}

\DeclareMathOperator{\covol}{covol}

\DeclareMathOperator{\disc}{disc}

\DeclareMathOperator{\Ram}{Ram}

\DeclareMathOperator{\Gal}{Gal}

\DeclareMathOperator{\Norm}{Norm}

\DeclareMathOperator{\PSL}{PSL}

\DeclareMathOperator{\Reg}{Reg}

\DeclareMathOperator{\tr}{tr}

\newtheorem{thm}{Theorem}[section]
\newtheorem{cor}[thm]{Corollary}

\newtheorem{prop}[thm]{Proposition}
\newtheorem{lem}[thm]{Lemma}
\theoremstyle{definition}

\theoremstyle{remark}

\newtheorem{proposition}{Proposition}[section]
\def\1{\mathbf{1}}
\def\disc{\mathrm{disc}}
\def\Prob{\mathbf{Prob}}

\theoremstyle{plain}
\newtheorem{mainthm}{Theorem}

\theoremstyle{remark}

\theoremstyle{plain}
\newtheorem{theorem}[proposition]{Theorem}

\def\C{\mathbf{C}}
\def\Q{\mathbf{Q}}

\def\pp{\mathfrak{p}}

\def\Qq{\curly{Q}}
\def\Pp{\curly{P}}

\def\R{\mathbf{R}}

\def\Q{\mathbf{Q}}
\def\Z{\mathbf{Z}}

\def\1{\mathbf{1}}
\def\disc{\mathrm{disc}}
\def\Gal{\mathrm{Gal}}

\newcommand{\frakp}{\mathfrak{p}}

\newcommand{\abs}[1]{\left\vert#1\right\vert}
\newcommand{\set}[1]{\left\{#1\right\}}

\makeatletter
\def\moverlay{\mathpalette\mov@rlay}
\def\mov@rlay#1#2{\leavevmode\vtop{%
   \baselineskip\z@skip \lineskiplimit-\maxdimen
   \ialign{\hfil$\m@th#1##$\hfil\cr#2\crcr}}}
\newcommand{\charfusion}[3][\mathord]{
    #1{\ifx#1\mathop\vphantom{#2}\fi
        \mathpalette\mov@rlay{#2\cr#3}
      }
    \ifx#1\mathop\expandafter\displaylimits\fi}
\makeatother

\makeatletter
\let\@@pmod\pmod
\DeclareRobustCommand{\pmod}{\@ifstar\@pmods\@@pmod}
\def\@pmods#1{\mkern4mu({\operator@font mod}\mkern 6mu#1)}
\makeatother

\begin{document}

\begin{abstract}
In this paper we examine the relationship between the length spectrum and the geometric genus spectrum of an arithmetic hyperbolic $3$-orbifold $M$.
In particular we analyze the extent to which the geometry of $M$ is determined by the closed geodesics coming from finite area totally geodesic surfaces. Using techniques from analytic number theory, we address the following problems: {\em Is the commensurability class of an arithmetic hyperbolic $3$-orbifold determined by the lengths of closed geodesics lying on totally geodesic surfaces?, Do there exist arithmetic hyperbolic $3$-orbifolds whose ``short'' geodesics do not lie on any totally geodesic surfaces?}, and {\em Do there exist arithmetic hyperbolic $3$-orbifolds whose ``short'' geodesics come from distinct totally geodesic surfaces?}
\end{abstract}

\maketitle


\vspace{-2pc}

\section{Introduction}

Over the past several years there have been a number of papers analyzing the extent to which the geometry of a finite volume orientable hyperbolic $3$-manifold $M$ is determined by its collections of proper totally geodesic subspaces. Two collections which have proven particularly important are the \textit{length spectrum} $\mathcal{L}(M)$ of $M$, which is the set of lengths of closed geodesics on $M$ considered with multiplicity, and the \textit{geometric genus spectrum} $\mathcal{GS}(M)$ of $M$ (see \cite{McReid}), which is the set of isometry classes of finite area,
properly immersed, totally geodesic surfaces of $M$ considered up to free homotopy. Neither of these sets determine the isometry class of $M$ (see \cite{Vi,McReid}). It is known however that both sets determine the commensurability class of $M$ whenever $M$ is arithmetic \cite{chinburg-geodesics,McReid}.
Recently there have been attempts to understand these phenomena in a more quantitative manner.

\textit{How many lengths does one need to determine commensurability?}
For any $x\in\textbf{R}_{>0}$ we refer to the set $\{ \ell\in\mathcal{L}(M) : \ell < x\}$ as the \textit{level $x$ initial length spectrum} of $M$. In \cite{LMPT} it was shown that the commensurability class of a compact arithmetic hyperbolic $3$-manifold of volume $V$ is determined by its level $x$ initial length spectrum for any $x>ce^{\log(V)^{\log(V)}}$ for some absolute constant $c>0$. Conversely, Millichap \cite{Mill} has shown that for any $n\in \Z_{>0}$ there exist infinitely many pairwise noncommensurable hyperbolic $3$-manifolds whose length spectra begin with the same $n$ lengths.

\textit{How many surfaces of bounded area does one need to determine commensurability?}
It was shown in \cite{LMPT} that there is a constant $c>0$ such that the commensurability class of a compact arithmetic hyperbolic $3$-orbifold $M$ having volume $V$ is determined by the surfaces in $\mathcal{GS}(M)$ which have area less than $e^{cV}$, provided of course that $\mathcal{GS}(M)\neq \emptyset$.

Up to this point, there has been no analysis of the interplay between the initial length spectra and geometric genus spectra.
It is the goal of this paper to initiate a broader project of carrying out such an analysis.

Every compact totally geodesic surface $S\in \mathcal{GS}(M)$ contributes closed geodesics to $\mathcal{L}(M)$.
We define the \textit{totally geodesic length spectrum} $\mathcal{L}^{TG}(M)$ of $M$ to be the set of lengths of closed geodesics of $M$ coming from finite area, totally geodesic surfaces. A natural question therefore suggests itself: \textit{Does the totally geodesic length spectrum of $M$ determine its commensurability class?}

\begin{mainthm}\label{theorem:comm}
Let $M$ be an arithmetic hyperbolic $3$-manifold for which $\mathcal{GS}(M)\neq \emptyset$. The commensurability class of $M$ is determined by $\mathcal{L}^{TG}(M)$ along with any geodesic length in $\mathcal{L}(M)$ associated to an element of $\pi_1(M)$ which is loxodromic but not hyperbolic.
\end{mainthm}

One might similarly ask for the distribution of the set $\mathcal{L}^{TG}(M)$ within $\mathcal{L}(M)$. As a first step we consider the following question:
\textit{Do there exist arithmetic hyperbolic 3-orbifolds with many ``short'' geodesics not lying on totally geodesic surfaces?}

\begin{mainthm}\label{theorem:theorem1}
For every integer $n\geq 1$ there is a positive real number $x_0$ such that the number of pairwise incommensurable arithmetic hyperbolic $3$-orbifolds $M$ of volume less than $V$ satisfying
\begin{enumerate}
\item $M$ has infinitely many totally geodesic surfaces, and
\item there are at least $n$ primitive geodesics on $M$ of length less than $x_0$, none of which lie on a totally geodesic surface
\end{enumerate}
is $\gg V^{1/2}/(\log{V})^{1-\frac{1}{2^{2n+1}}}$ as $V\to\infty$, where the implied constant depends on $n$. In fact, for any $\epsilon>0$ there is a constant $c_\epsilon>0$ such that one may take $x_0=c_\epsilon n^{16+\epsilon}$, hence for sufficiently large $n$ one may take $x_0=n^{17}$.
\end{mainthm}

Suppose that $M$ is an arithmetic hyperbolic $3$-orbifold with geodesics of lengths $\ell_1,\dots,\ell_n$. Theorems 4.9 and 4.10 of \cite{LMPT} provide lower bounds for the number of incommensurable arithmetic hyperbolic $3$-orbifolds with volume less than $V$ and which all contain geodesics of lengths $\ell_1,\dots,\ell_n$. One might therefore expect to be able to deduce Theorem \ref{theorem:theorem1} from these results. There is a nuance however; none of the arithmetic hyperbolic $3$-orbifolds counted by Theorems 4.9 and 4.10 of \cite{LMPT} are guaranteed to contain any totally geodesic surfaces. Indeed, a ``random'' arithmetic hyperbolic $3$-manifold is likely not to contain any totally geodesic surfaces. In order to circumvent these difficulties we use the well-known construction of arithmetic hyperbolic $3$-manifolds from quaternion algebras defined over number fields to reduce the proof of Theorem \ref{theorem:theorem1} to a series of problems which can be handled using techniques from analytic number theory. Among the techniques that we employ are mean value estimates for multiplicative functions and the linear sieve.

Our final result concerns a question dual to the one addressed in Theorem \ref{theorem:theorem1}: \textit{Do there exist arithmetic hyperbolic $3$-orbifolds whose ``short'' geodesics come primarily from distinct totally geodesic surfaces?}

\begin{mainthm}\label{theorem:theorem2}
For every integer $n\geq 1$ there are positive real numbers $x_0, x_1$ such that the number of pairwise incommensurable arithmetic hyperbolic $3$-orbifolds of volume less than $V$ with at least $n$ geodesics of length at most $x_0$ lying on pairwise incommensurable totally geodesic surfaces of area at most $x_1$ is $\gg n^{-cn} V^{\frac{2}{3}}$ as $V\to\infty$, for some constant $c>0$. In fact, there are constants $c_1,c_2,c_3>0$ such that we may take $x_0=c_1\left(n\log(2n)\right)^{2}$ and
$x_1=c_2n^{c_3n}$.
\end{mainthm}

The proof of Theorem \ref{theorem:theorem2} requires us to understand how often a given finite set of primes splits in a prescribed way, across the family of quadratic number fields. While this particular problem could be treated by ad hoc methods, recent work has shown that such problems are best understood in the broader context of the field of arithmetic statistics. In our work, we appeal to powerful recent results of Wood \cite{wood10}, which address such statistical prime splitting questions in great generality.



\section{Notation}

Throughout this paper $k$ will denote a number field with ring of integers $\mathcal O_k$, degree $n_k$ and discriminant $\Delta_k$. We will denote by $r_1(k)$ the number of real places of $k$ and by $r_2(k)$ the number of complex places of $k$. We will denote by $\Pp_k$ the set of prime ideals of $k$ and, given $\mathfrak{p}\in\Pp_k$, by $\abs{\mathfrak{p}}$ the norm of $\mathfrak{p}$. The Dedekind zeta function of $k$ will be denoted by $\zeta_k(s)$ and the regulator of $k$ by $\Reg_k$. If $L/k$ is a finite extension of number fields we will denote by $\Delta_{L/k}$ the relative discriminant.

Given a number field $k$ and quaternion algebra $B$ over $k$, we denote by $\Ram(B)$ the set of primes (possibly infinite) of $k$ which ramify in $B$ and by $\Ram_f(B)$ (respectively $\Ram_\infty(B)$) the subset of $\Ram(B)$ consisting of those finite (respectively infinite) primes of $k$ ramifying in $B$. The discriminant $\disc(B)$ of $B$ is defined to be the product of all primes (possibly infinite) in $\Ram(B)$. We define $\disc_f(B)$ similarly.

We denote by $\textbf{H}^2$ and $\textbf{H}^3$ real hyperbolic space of dimension $2$ and $3$. Throughout this paper $M$ will denote an arithmetic hyperbolic $3$-manifold and $\Gamma=\pi_1(M)$ the associated arithmetic lattice in $\PSL_2(\textbf{C})$. We will refer to lattices in $\PSL_2(\textbf{C})$ as Kleinian and lattices in $\PSL_2(\textbf{R})$ as Fuchsian.

We will make use of standard analytic number theory notation. We will interchangeably use the Landau ``Big Oh'' notation, $f=O(g)$, and the Vinogradov notation, $f\ll g$, to indicate that there exists a constant $C>0$ such that $|f| \leq C |g|$. We write $f\sim g$ if $\lim_{x\to\infty}\frac{f(x)}{g(x)}=1$ and $f=o(g)$ if $\lim_{x\to\infty}\frac{f(x)}{g(x)}=0$. Finally, throughout this paper we denote by $\log(x)$ the natural logarithm.


\section{Constructing arithmetic hyperbolic $2$- and $3$-manifolds}

In this section we briefly review the construction of arithmetic lattices in $\PSL_2(\textbf{R})$ and $\PSL_2(\textbf{C})$. For a more detailed treatment we refer the reader to Maclachlan and Reid \cite{MR}.

We begin by describing the construction of arithmetic hyperbolic surfaces. Let $k$ be a totally real field and $B$ a quaternion algebra defined over $k$ which is unramified at a unique real place $v$ of $k$. Let $\mathcal O$ be a maximal order of $B$, $\mathcal O^1$ be the multiplicative subgroup consisting of those elements of $\mathcal O^*$ having reduced norm $1$ and $\Gamma_{\mathcal O}^1$ be the image in $\PSL_2(\textbf{R})$ of $\mathcal{O}^1$ under the identification $B_v=B\otimes_k k_v\cong \mathrm{M}_2(\textbf{R})$. The group $\Gamma_{\mathcal O}^1$ is a discrete finite coarea subgroup of $\PSL_2(\textbf{R})$ which is cocompact whenever $B$ is a division algebra. If $\Gamma$ is a lattice in $\PSL_2(\textbf{R})$ then we say that $\Gamma$ is an {\em arithmetic Fuchsian group} if there exist $k,B,\mathcal O$ such that $\Gamma$ is commensurable in the wide sense with $\Gamma_{\mathcal O}^1$.

The construction of arithmetic hyperbolic $3$-manifolds is extremely similar. Let $k$ be a number field which has a unique complex place $v$ and $B$ be a quaternion algebra over $k$ in which all real places of $k$ ramify. Let $\mathcal O$ be a maximal order of $B$ and $\Gamma_{\mathcal O}^1$ be the image in $\PSL_2(\textbf{C})$ of $\mathcal{O}^1$ under the identification $B_v=B\otimes_k k_v\cong \mathrm{M}_2(\textbf{C})$. The group $\Gamma_{\mathcal O}^1$ is a discrete finite covolume subgroup of $\PSL_2(\textbf{C})$ which is cocompact whenever $B$ is a division algebra. If $\Gamma$ is a lattice in $\PSL_2(\textbf{C})$ then we say that $\Gamma$ is an {\em arithmetic Kleinian group} if there exist $k,B,\mathcal O$ such that $\Gamma$ is commensurable in the wide sense with $\Gamma_{\mathcal O}^1$. An arithmetic Kleinian group $\Gamma$ is said to be {\em derived from a quaternion algebra} if it is contained in a group of the form $\Gamma_{\mathcal O}^1$.

Given arithmetic lattices $\Gamma_1,\Gamma_2$ arising from quaternion algebras $B_1, B_2$ defined over $k_1, k_2$, we note that $\Gamma_1$ and $\Gamma_2$ are commensurable in the wide sense if and only if $k_1\cong k_2$ and $B_1\cong B_2$ (see \cite[Theorem 8.4.1]{MR}). We will make crucial use of this fact many times in this paper.


\section{Geometry Background}

An element $\gamma\in \PSL_2(\C)$ is \textit{loxodromic} if $\mathrm{tr}\, \gamma\in \C\setminus[-2,2]$ and is \textit{hyperbolic} if $\mathrm{tr}\, \gamma\in \R\setminus[-2,2]$.
Geometrically, a loxodromic element $\gamma$ acts on $\mathbf{H}^3$ by translating along, and possibly rotating around,  an axis that we denote $A_\gamma$.

We now give a characterization of a hyperbolic element in terms of its action on the totally geodesic hyperplanes containing its axis.

\begin{lem}\label{lem:list}
If $\gamma\in \PSL_2(\C)$ is a loxodromic element, then the following are all equivalent:
\begin{enumerate}
\item $\gamma$ is hyperbolic,
\item $\gamma$ has only real eigenvalues,
\item $\gamma$ stabilizes and preserves the orientation of a totally geodesic hyperplane containing $A_\gamma$,
\item $\gamma$ stabilizes and preserves the orientation of all totally geodesic hyperplanes containing $A_\gamma$.
\end{enumerate}
\end{lem}

\begin{proof}
Let $\lambda_1$ and $\lambda_2$ denote the eigenvalues of $\gamma$.  It is an immediate consequence of the equations $\lambda_1\lambda_2=1$ and $\lambda_1+\lambda_2= \mathrm{tr}\, \gamma$ that $(i)\Leftrightarrow(ii)$.
Assuming $(i)$, $\gamma$ parallel transports tangent vectors along $A_\gamma$.
Since a totally geodesic hyperplane $\widetilde{S}\subset \mathbf{H}^3$ containing $A_\gamma$ is completely determined by $A_\gamma$ and a normal direction,
it follows that $(i)\Rightarrow(iv)\Rightarrow(iii)$.
Lastly, assuming $(iii)$, it follows that as $\gamma$ translates along  $A_\gamma$, it has no rotation, and hence must be hyperbolic, implying $(i)$.
\end{proof}

\begin{lem}\label{lem:nosharedaxis}
Let $\Gamma$ be an arithmetic Kleinian group and $\gamma\in \Gamma$ be a loxodromic element.
If $A_\gamma$ is an axis of a hyperbolic element in $\Gamma$, then
 there exists an integer $n\ge1$ such that $\gamma^n$ hyperbolic.
\end{lem}

\begin{proof}
Suppose that $A_\gamma$ is also the axis of a hyperbolic element $\delta\in \Gamma$.
Since $\Gamma$ acts discretely, it must be the case that the translation lengths of $\gamma$ and $\delta$ are rational multiples of one another.
In particular, there exist positive integers $a$ and $b$ and an elliptic element $\varepsilon\in \Gamma$ such that $\gamma^b=\varepsilon\delta^a$.
Again by the discreteness of $\Gamma$, elliptic elements in $\Gamma$ are torsion, hence $\varepsilon$ has finite order $c$.
Therefore  $\gamma^{bc}=\delta^{ac}$, which is hyperbolic.
\end{proof}

The following result will play a crucial role in the proofs of Theorems \ref{theorem:theorem1} and \ref{theorem:theorem2}.

\begin{prop}\label{prop:notreal}
Let $\Gamma$ be an arithmetic Kleinian group and $\gamma\in \Gamma$ be a loxodromic element.
If there does not exist an integer $n\ge1$ such that $\gamma^n$ is hyperbolic,
then the geodesic associated to $\gamma$ lies in no finite area, totally geodesic surface of $\mathbf{H}^3/\Gamma$.
\end{prop}

\begin{proof}
We prove the contrapositive.
Let $c_\gamma$ be the closed geodesic associated to $\gamma$ and suppose that $c_\gamma$ lies in a finite area, totally geodesic surface $S\subset \mathbf{H}^3/\Gamma$.
Then the axis $A_\gamma$ lies in a totally geodesic hyperplane $\widetilde{S}\subset \mathbf{H}^3$ that covers $S$.
Since $S$ is finite area and hyperbolic, $\Lambda:=\mathrm{Stab}_{\Gamma}(\widetilde{S})$ is an arithmetic Fuchsian subgroup of $\Gamma$, and hence there exists some hyperbolic element $\delta\in \Lambda$ with axis $A_\gamma$.
By Lemma \ref{lem:nosharedaxis}, the result then follows.
\end{proof}

We note that it is a result of Long and Reid \cite{LR} that the fundamental group of any finite volume orientable hyperbolic $3$-manifold has infinitely many loxodromic elements, no power of which is hyperbolic.


\section{Counting quaternion algebras over quadratic fields}

Let $k$ be a number field with a unique complex place and $L_1,\dots, L_n$ be quadratic extensions of $k$ with images under complex conjugation $L_1', \dots, L_n'$. Suppose that $[L_1 \cdots L_n L_1' \cdots L_n': k] = 2^{2n}$. In this section we will consider the case in which $k$ is an imaginary quadratic field and will count the number of quaternion algebras $B$ over $k$ which admit embeddings of all of the fields $L_i$ and are also of the form $B_0\otimes_{\Q}k$ for infinitely many rational indefinite quaternion algebras $B_0$. While the latter condition may seem arbitrary, it ensures that the Kleinian groups arising from $B$ contain Fuchsian subgroups arising from $B_0$ and hence will allow us to construct arithmetic hyperbolic $3$-manifolds containing infinitely many totally geodesic surfaces \cite[Theorem 9.5.4]{MR}.


\begin{theorem}\label{theorem:tgs}
Let $k$ be a number field with a unique complex place and suppose that the maximal totally real subfield $k^+$ of $k$ satisfies $[k:k^+]=2$. Suppose $B^+$ is a quaternion algebra over $k^+$ ramified at all real places of $k^+$ except at the place under the complex place of $k$. Then $B \cong B^+ \otimes_{k^+} k$ if and only if $\Ram_f(B)$ consists of $2r$ places $\set{\mathfrak{P}_{i,j}}_{1\le i\le r,\,1\le j\le 2}$ satisfying
\[ \mathfrak{P}_{1,j}\cap \mathcal{O}_{k^+} = \mathfrak{P}_{2,i} \cap \mathcal{O}_{k^+} = \pp_i, \]
$\set{\pp_1,\dots,\pp_r} \subset \Ram_f(B^+)$ with $\Ram_f(B^+)\setminus \set{\pp_1,\dots ,\pp_r}$ consisting of primes in $\mathcal{O}_{k^+}$ which are inert or ramified in $k/k^+$.
\end{theorem}

When the above conditions on $\Ram(B)$ are satisfied, there are infinitely many algebras $B^+/k^+$ such that $B \cong B^+ \otimes_{k^+} k$. In particular, any arithmetic hyperbolic 3--orbifold constructed from $B$ will contain infinitely many primitive, totally geodesic, incommensurable surfaces.

\begin{theorem}\label{theorem:quaternionalgebraswithspecifieddiscriminantshape}
 Let $k$ be an imaginary quadratic field. Let $L_1, \dots, L_n$ be quadratic extensions of $k$, and let $L_1', \dots, L_n'$ be their images under complex conjugation. Suppose that $[L_1 \cdots L_n L_1' \cdots L_n': \Q] = 2^{2n}$. The number of quaternion algebras $B$ over $k$ which admit embeddings of all of the $L_i$, have discriminants of the form $\disc(B) = p_1 \cdots p_r \mathcal O_k$ where the $p_i$ are rational primes split in $k$, and have $|\disc_f(B)| < x$, is asymptotically
\[ C(k,L_1,\dots,L_n)  x^{1/2}/(\log{x})^{1-\frac{1}{2^{2n+1}}},\]
as $x\to\infty$. Here $C(k,L_1, \dots, L_n)$ is a positive constant.
\end{theorem}

We need a lemma on mean values of multiplicative functions. The next result appears in more precise form in work of Spearman and Williams \cite[Proposition 5.5]{SW}.

\begin{prop}\label{prop:odoni} Let $f$ be a multiplicative function satisfying $0 \le f(n) \le 1$ for all positive integers $n$. Suppose that there are positive constants $\tau$ and $\beta$ with
\[ \sum_{p \le x} f(p) = \tau \frac{x}{\log{x}} + O\left(\frac{x}{(\log{x})^{1+\beta}}\right). \]
As $x\to\infty$,
\[ \sum_{n \le x} f(n) \sim C_f x (\log{x})^{\tau-1} \]
for a certain positive constant $C_f$.
\end{prop}

This has the following immediate consequence.

\begin{cor}\label{cor:odoni} Let $\Pp$ be a set of primes. Suppose that there are positive constants $\tau$ and $\beta$ with
\[ \sum_{\substack{p \le x \\ p \in \Pp}} 1 = \tau \frac{x}{\log{x}} + O\left(\frac{x}{(\log{x})^{1+\beta}}\right).  \]
As $x\to\infty$, the number of squarefree integers $d\le x$ composed entirely of primes from $\Pp$ is asymptotic to $C_{\Pp} x (\log{x})^{\tau-1}$ for a certain positive constant $C_{\Pp}$.
\end{cor}

\begin{proof} Apply Proposition \ref{prop:odoni} to the characteristic function of these $d$.
\end{proof}

\begin{proof}[Proof of Theorem \ref{theorem:quaternionalgebraswithspecifieddiscriminantshape}] If the $p_i$ are distinct primes split in $k$ and $\disc(B) = p_1 \cdots p_r \mathcal O_k$, then $B$ admits embeddings of all of the $L_i$ if and only if each $p_i \in \Pp$, where
\[ \Pp=\{p: p\text{ splits in $k$, every $\pp \mid p$ is nonsplit in every $L_i$}\}. \]
Thus, we can count the number of possibilities for $B$ by counting the number of squarefree $d$ composed entirely of primes from $\Pp$.

Observe that if $p$ splits in $k$ as $\pp \pp'$, then $\pp$ splits in one of the fields $L_i$ if and only if $\pp'$ splits in the corresponding $L_i'$. Thus, $p$ belongs to $\Pp$ precisely when $p$ splits as a product of two primes neither of which split in any of the $L_i$ or $L_i'$. Let
\[ \Qq= \{\pp: \pp \text{ prime of $k$ not split in any $L_i$ or $L_i'$}\}. \]
Let $L$  be the compositum of the $L_i$ and $L_i'$. Since $[L:k] = 2^{2n}$, the Galois group of $L/k$ is canonically isomorphic to the direct sum of the groups $\Gal(L_i/k)$ and $\Gal(L_i'/k)$. The Chebotarev density theorem now implies that the set $\Qq$ has density $2^{-2n}$, and the quantitative form of the theorem appearing in \cite{Schulze} shows that as $X\to\infty$,
\begin{equation}\label{eq:qqcount} \#\{\pp \in \Qq: \Norm_{k/\Q}(\pp) \le X\} = \frac{1}{2^{2n}} \frac{X}{\log{X}} + O\left(\frac{X}{(\log{X})^2}\right). \end{equation}
(We  allow the implied constant to depend on $k$ and the $L_i$.)  Let
\[ \Qq'= \{ \pp \in \Qq: \pp \text{ absolute degree $1$}, \pp \nmid \Delta_k \}. \]
The number of elements of $\Qq \setminus \Qq'$ with norm not exceeding $X$ is $O(X^{1/2})$. Thus, the estimate \eqref{eq:qqcount} continues to hold if $\Qq$ is replaced by $\Qq'$. The norm from $K$ down to $\Q$ induces a $2$-to-$1$ map from $\Qq'$ onto $\Pp$. Thus, as $X\to\infty$,
\[ \sum_{\substack{p \le X \\ p \in \Pp}} 1= \frac{1}{2^{2n+1}} \frac{X}{\log{X}} + O\left(\frac{X}{(\log{X})^2}\right). \]

By Corollary \ref{cor:odoni}, the number of squarefree $d \le X$ composed of primes from $\Pp$ is asymptotically $C X/(\log{X})^{1-\frac{1}{2^{2n+1}}}$, for a certain constant $C$. Since $|d\mathcal O_k| = d^2$, we obtain the theorem upon taking $X = x^{1/2}$.
\end{proof}

\section{Proof of Theorem \ref{theorem:comm}}

Let $\Gamma$ denote the fundamental group of $M$. As in Reid's proof \cite{R} that the commensurability class of an arithmetic hyperbolic surface is determined by its length spectrum, it suffices to show that the invariant trace field $k$ and invariant quaternion algebra $B$ from which $\Gamma$ arises are determined by $\mathcal{L}^{TG}(M)$ and any geodesic length $\ell(\gamma)$ associated to an element $\gamma\in\Gamma$ which is loxodromic but not hyperbolic.

Let $\gamma\in\Gamma$ be as above and $\Gamma^{(2)}$ be the subgroup of $\Gamma$ generated by squares. We will show that $k$ is determined by $\ell(\gamma^2)=2\ell(\gamma)$. It follows from the formula \[ \cosh(\ell(\gamma^2)/2)=\pm\tr(\gamma^2)/2\] that $\ell(\gamma)$ determines $\tr(\gamma^2)$ (up to a sign). As $\Gamma^{(2)}$ is derived from $B$ in the sense that there is a maximal order $\mathcal O$ of $B$ such that $\Gamma\subset \Gamma_{\mathcal O}^1$ (see \cite[Chapter 3]{MR}), Lemma 2.3 of \cite{chinburg-geodesics} shows that $k=\Q(\tr(\gamma^2))$.

It remains to show that $\mathcal{L}^{TG}(M)$ determines the isomorphism class over $k$ of the quaternion algebra $B$. Let $\ell(\gamma')\in\mathcal{L}^{TG}(M)$. The argument above shows that $\ell(\gamma')$ determines $\tr(\gamma'^2)$ (up to a sign), and another application of Lemma 2.3 of \cite{chinburg-geodesics} allows us to deduce that the maximal totally real subfield $k^+$ of $k$ satisfies $k^+=\Q(\tr(\gamma'^2))$ and $[k:k^+]=2$.

Recall now that Theorem \ref{theorem:tgs} shows that there are primes $\mathfrak{p}_1,\dots,\mathfrak{p}_r$ of $k^+$, all of which split in $k/k^+$, such that $\disc_f(B)=\mathfrak{p}_1\cdots\mathfrak{p}_r\mathcal{O}_k$. The algebra $B$ is ramified at all real primes of $k$, hence it suffices to show that $\mathcal{L}^{TG}(M)$ determines $\mathfrak{p}_1,\dots,\mathfrak{p}_r$.

Consider the set \[ S=\{ k^+(\lambda(\gamma'^2)) : \ell(\gamma')\in\mathcal{L}^{TG}(M) \}\] of quadratic extensions of $k^+$. This is precisely the set of maximal subfields of quaternion algebras $B^+$ over $k^+$ such that $B^+\otimes_{k^+} k\cong B$. Given a field $L\in S$, denote by $\mathrm{Spl}(L/k^+)$ the set of primes in $k^+$ which split in the extension $L/k^+$. We claim that \[ \bigcap_{L\in S} \left[ \Pp_{k^+}\setminus\mathrm{Spl}(L/k^+)\right] = \{\mathfrak{p}_1,\dots, \mathfrak{p}_r \}.\] We begin by observing that none of the primes $\mathfrak{p}_i$ split in any of the extensions $L/k^+$ where $L\in S$. Indeed, this follows from Theorem \ref{theorem:tgs} along with the Albert-Brauer-Hasse-Noether theorem, which in this context states that the field $L$ embeds into $B^+$ if and only if no prime of $k^+$ which ramifies in $B^+$ splits in $L/k^+$. Suppose now that $\mathfrak{p}\in \Pp_{k^+}\setminus \{\mathfrak{p}_1,\dots,\mathfrak{p}_r\}$. We will show that there exists a field $L\in S$ such that $\mathfrak{p}\in \mathrm{Spl}(L/k^+)$. Let $\mathfrak{p}'$ be a finite prime of $k^+$ which is inert in $k/k^+$ and $B^+$ be the quaternion algebra over $k^+$ which is ramified at all real places of $k^+$ except for the one lying below the complex place of $k$ along with $\{\mathfrak{p}_1,\dots, \mathfrak{p}_r \}$ if $r+r_1(k^+)-1$ is even and $\{\mathfrak{p}_1,\dots, \mathfrak{p}_r \}\cup \{\mathfrak{p}'\}$ if $r+r_1(k^+)-1$ is odd. Theorem \ref{theorem:tgs} shows that $B^+\otimes_{k^+}k\cong B$, hence every maximal subfield of $B^+$ lies in $S$. That there exists a maximal subfield of $B^+$ in which $\mathfrak{p}$ splits now follows from the Grunwald-Wang theorem \cite[Theorem 32.18]{Reiner}, which implies the existence of quadratic field extensions of $k^+$ satisfying prescribed local splitting behavior at any finite number of primes. This completes the proof of Theorem \ref{theorem:comm}.

\section{Proof of Theorem \ref{theorem:theorem1}}

\begin{prop}\label{proposition:constructionoffields}
Let $k$ be an imaginary quadratic field of discriminant $\Delta_k$, and let $n$ be a positive integer. Let $\epsilon > 0$. There exist quadratic extensions $L_1,\dots,L_n$ of $k$ such that
\begin{enumerate}
\item none of the quartic extensions $L_i/\Q$ are Galois,

\item the compositum of the fields $L_i$ and $L_i'$ has degree $2^{2n}$ over $k$, where $L_i'$ is the image in $\C$ of $L_i$ under complex conjugation, and

\item the absolute value of the discriminants of the $L_i$ all satisfy $\abs{\Delta_{L_i}}\leq c(k,\epsilon) n^{8+\epsilon}$, where $c(k,\epsilon)$ is a constant depending on $k$ and $\epsilon$.
\end{enumerate}	
\end{prop}

\begin{proof}[Proof of Proposition \ref{proposition:constructionoffields}] 

Let $p_1, p_2, \dots, p_n$ denote the first $n$ odd primes that split in $k$. The splitting condition amounts to restricting the $p_i$ to a certain half of the coprime residue classes modulo $\Delta_k$, and so the prime number theorem for progressions shows that
\[ p_n < c_1 n\log{(2n)},\]
where $c_1$ is a constant depending only on $k$.

Now $\Delta_k$ is a square modulo each $p_i$, and by Hensel's lemma, $\Delta_k + p_i$ is a square modulo $p_i^2$. For each $i=1, 2, \dots, n$, we will choose $x_i$ such that
\begin{equation}\label{eq:xiconditions} x_i^2 \equiv \Delta_k+p_i \pmod{p_i^2}, \quad\text{while for every $1\le j \ne i \le n$}, \quad x_i^2\not\equiv \Delta_k \pmod{p_j}. \end{equation}
To this end, let $r_i$ be the smallest nonnegative integer with $r_i^2 \equiv \Delta_k + p_i \pmod{p_i^2}$, so that $0 \le r_i < p_i^2$. We will take $x_i = r_i + p_i^2 t_i$ for a nonnegative integer $t_i$. To choose $t_i$ as small as possible, we apply a lower bound sieve method. The second half of \eqref{eq:xiconditions} will hold as long as $t_i$ avoids a certain two residue classes modulo each $p_j$, with $j \ne i$. Since the $p_j$ are restricted to the set of primes splitting in $k$ --- which form a set of density $\frac{1}{2}$ --- we have a sieve problem of dimension $2 \cdot \frac{1}{2} = 1$. The linear sieve (see, e.g., \cite[Theorem 7.1, p. 81]{DH}) implies that we may take each $t_i < c_2  p_n^{2+\epsilon}$, for a constant $c_2$ depending on $k$ and $\epsilon$. Thus,
\[ x_i = r_i + p_i^2 t_i \le p_i^2 + p_i^2 \cdot c_2 p_n^{2+\epsilon} \le c_3 p_n^{4+\epsilon} \]
for a certain $c_3 =c_3(k,\epsilon)$.

We now define
\[ L_i = k(\sqrt{x_i + \sqrt{\Delta_k}}), \]
so that the image of $L_i$ under complex conjugation is
\[ L_i' = k(\sqrt{x_i - \sqrt{\Delta_k}}). \]
By our choice of $x_i$, the norm from $k$ to $\Q$ of $x_i + \sqrt{\Delta_k}$ is divisible by $p_i$ but not $p_i^2$. Hence,
$x_i + \sqrt{\Delta_k}$ is not a square in $k$, and so $L_i$ is indeed a quadratic extension of $k$.

We now check that $(i)$, $(ii)$, and $(iii)$ hold.

In fact, $(i)$ is a consequence of $(ii)$, since $(ii)$ implies that $L_i$ and its conjugate field $L_i'$ are distinct. To prove $(ii)$, we appeal to elementary results about the splitting of primes in relative quadratic extensions. Let $\pp_i = (p_i, x_i + \sqrt{\Delta_k}) \subset \mathcal{O}_k$. Then $\pp_i$ is a prime of $k$ lying above the rational prime $p_i$, and $\Norm_{k/\Q}(\pp_i)= p_i$. Moreover, since $\pp_i \mid x_i + \sqrt{\Delta_k}$ and
$\pp_i^2 \nmid x_i + \sqrt{\Delta_k}$, the prime $\pp_i$ ramifies in $L_i$ (see, e.g., \cite[Theorem 118, pp. 134--135]{hecke}). We claim that $L_i$ is the only field among $L_1, \dots, L_n$, $L_1', \dots, L_n'$ in which $\pp_i$ ramifies. For this, it is enough to show (by the same theorem from \cite{hecke}) that for all $j \ne i$, $\pp_i \nmid x_j \pm \sqrt{\Delta_k}$ and that $\pp_i \nmid x_i - \sqrt{\Delta_k}$. The first half of this is clear, since $p_i = \Norm_{k/\Q}(\pp_i) \nmid x_j^2 - \Delta_k$. Turning to the second half, observe that if $\pp_i \mid x_i - \sqrt{\Delta_k}$, then
\[ \pp_i \mid (x_i+\sqrt{\Delta_k})-(x_i-\sqrt{\Delta_k}) = 2\sqrt{\Delta_k}, \]
contradicting that $\pp_i$ lies above an odd prime not dividing $\Delta_k$. We deduce that $L_i$ is not contained in the compositum of the $2n-1$ other fields. Applying complex conjugation, we see that the same holds for $L_i'$. Now  $(ii)$ follows immediately.

Turning to $(iii)$, notice that $L_i = \Q(\theta_i)$, where $\theta_i = \sqrt{x_i + \sqrt{\Delta_k}}$.
The minimal polynomial of $\theta_i$ over $\Q$ is $f_i(T) = (T^2-x_i)^2-\Delta_k = T^4 - 2x_i T^2 + (x_i^2-\Delta_k)$, and
\[ \disc(f_i(T))= 256 (x_i^2-\Delta_k) \Delta_k^2. \]
Since $\Delta_{L_i} \mid \disc(f_i(T))$,
\begin{align*} |\Delta_{L_i}| \le 256(|x_i|^2+|\Delta_k|) |\Delta_k|^2 &\le 256(c_3^2 p_n^{8+2\epsilon} + |\Delta_k|) |\Delta_k|^2 \\ &\le 256(c_3^2 (c_1 n\log(2n))^{8+2\epsilon} + |\Delta_k|) |\Delta_k|^2 \le c_4 n^{8+3\epsilon},
\end{align*}
for a certain $c_4 = c_4(k,\epsilon)$. Replacing $\epsilon$ with $\epsilon/3$, we obtain $(iii)$.
\end{proof}

We now prove Theorem \ref{theorem:theorem1}.

Fix an imaginary quadratic field $k$ of discriminant $\Delta_k$ and let $L_1,\dots, L_n$ be quadratic extensions of $k$ satisfying the conditions of Proposition \ref{proposition:constructionoffields}.


Let $B$ be a quaternion division algebra over $k$ into which all of the $L_i$ embed and whose discriminant is of the form $p_1\cdots p_r\mathcal O_k$ where $p_1,\dots, p_r$ are rational primes splitting in $k/\Q$. Finally, let $\mathcal O$ be a maximal order of $B$. We will construct, for each field $L_i$, a loxodromic element $\gamma_i\in \Gamma_{\mathcal{O}}^1$ such that $L_i=k(\lambda(\gamma_i))$ and whose associated geodesic has length which we can bound.


It follows from Dirichlet's unit theorem that there exists a fundamental unit $u_0\in\mathcal O_{L_i}^*$ such that $u_0^m\not\in \mathcal{O}_k^*$ for any $m\geq 1$. Let $\sigma\in\Gal(L_i/k)$ denote the nontrivial Galois automorphism of $L_i/k$ and define $u=u_0/\sigma(u_0)$. It is then clear that $\Norm_{L_i/k}(u)=1$ and $u^n\not\in \mathcal O_k^*$ for any $n\geq 1$. Work of Brindza \cite{brindza} and Hajdu \cite{hajdu} shows that $u_0$ maybe chosen so that the absolute logarithmic Weil height $h(u)$ of $u$ satisfies \[ h(u)\leq 2h(u_0) \leq 6^{n_{L_i}}n_{L_i}^{5n_{L_i}}\Reg_{L_i}=2^{44}3^4\Reg_{L_i}\leq 2^{44}3^4\abs{\Delta_{L_i}}^2, \] where the last inequality follows from \cite[Lemma 4.4]{LMPT}. Because the algebra $B/k$ must be ramified at at least two finite primes of $k$, Theorem 3.3 of \cite{Chinburg-Friedman-selectivity} implies that every maximal order of $B$ (so in particular $\mathcal O$) admits an embedding of the quadratic $\mathcal O_k$-order $\mathcal O_k[u]$. Let $\gamma_i$ denote the image of $u$ in $\Gamma_{\mathcal O}^1$. The logarithm of the minimal polynomial of $u$ is equal to $4h(u)$, hence Lemma 12.3.3 of \cite{MR} and Proposition \ref{proposition:constructionoffields} imply that for any $\epsilon>0$,
\[
\ell(\gamma_i)\leq 2^{47}3^4\abs{\Delta_{L_i}}^2\leq c'(k,\epsilon) n^{16+2\epsilon},
\]
where the constant $c'(k,\epsilon)$ depends only on $k$ and $\epsilon$. By construction the extension $L_i/\Q$ is not Galois for any $i$. Lemma 2.3 of \cite{chinburg-geodesics} therefore implies that $\lambda(\gamma_i)^m$ is not real for any $m\geq 1$. This allows us to deduce from Proposition \ref{prop:notreal} that the geodesic associated to $\gamma_i$ lies on no totally geodesic surface of $\textbf{H}^3/\Gamma_{\mathcal{O}}^1$.

Because the discriminant of $B$ is of the form $p_1\cdots p_r\mathcal O_k$ for rational primes $p_1,\dots, p_r$ which split in $k/\Q$, Theorems 9.5.4 and 9.5.5 of \cite{MR} imply that $\textbf{H}^3/\Gamma_{\mathcal{O}}^1$ contains infinitely many primitive, totally geodesic, pairwise incommensurable surfaces.

Borel's volume formula \cite{borel-commensurability} shows that
\begin{align*}
\covol(\Gamma_{\mathcal{O}}^1) &= \frac{ \abs{\Delta_k}^{\frac{3}{2}}\zeta_k(2)}{4\pi^2}\prod_{\frakp\mid\disc_f(B)}\left({\abs{\frakp}}-1\right)\\
&\leq c_k \abs{\disc_f(B)},
\end{align*}
where $c_k$ is a positive constant depending only on $k$. Theorem \ref{theorem:theorem1} now follows from Theorem \ref{theorem:quaternionalgebraswithspecifieddiscriminantshape} and the accompanying discussion upon fixing a choice of $k$ and replacing $\epsilon$ by $\epsilon/2$.
\section{Proof of Theorem \ref{theorem:theorem2}}

The proof uses the following form of Linnik's theorem on primes in arithmetic progressions.

\begin{prop}\label{prop:linnik} There is an absolute constant $L$ for which the following holds: For every pair of coprime integers $a$ and $q$ with $q\ge 2$, the $n$th prime $p \equiv a\pmod{q}$ satisfies
\[ p \le q^{L} \cdot n \log{(2n)}. \]
\end{prop}

\begin{proof} From \cite[Corollary 18.8, p. 442]{IK}, there are constants $c > 0$ and $L> 0$ with the property that
\begin{equation}\label{eq:IKlinnik} \#\{p \le x: p \equiv a\pmod{q}\} \ge \frac{c}{q^{1/2} \phi(q)} \frac{x}{\log{x}} \end{equation}
whenever $x \ge q^{L}$. Of course, this remains true if we increase the value of $L$, so we can assume that $L \ge 3$.

Let $A$ be a large constant, to be specified more precisely momentarily. To start with, assume $A \ge 2$. We apply \eqref{eq:IKlinnik} with $x= A q^{L} \cdot n \log(2n)$, noting that this choice certainly satisfies $x \ge q^L$. Now
\begin{align*} \log{x} &\le 2 \max\{\log(A \cdot q^L), \log(n\log(2n))\} \\
&\le 2 \max\{A^{1/2} q^{L/2}, 2\log{n}\}. \end{align*}
(We use here that $\log{t} < \sqrt{t}$ for all $t > 0$, and that $n\log(2n) < n^2$ for all natural numbers $n$.)
If this maximum is given by the first term, then the right-hand side of \eqref{eq:IKlinnik} is bounded below by
\[ \frac{c}{q^{1/2} \phi(q)} \cdot \frac{A q^L \cdot n\log(2n)}{2 A^{1/2} q^{L/2}} = \frac{c\cdot A^{1/2}}{2} \frac{q^{L/2}}{q^{1/2}\phi(q)} \cdot n \log(2n) \ge \frac{c \cdot A^{1/2} \log{2}}{2} \cdot n. \]
On the other hand, if the maximum is given by the second term, we obtain a lower bound of
\[  \frac{c}{q^{1/2} \phi(q)} \cdot \frac{A q^L \cdot n\log(2n)}{4\log{n}} \ge \frac{c A}{4} \cdot \frac{q^L}{q^{1/2} \phi(q)} n \ge \frac{cA}{4} n. \]
Choosing $A = \max\{2, 4/c, (2/(c\log{2}))^2\}$, we see that in either case the right-hand side of \eqref{eq:IKlinnik} is at least $n$. This proves the proposition except with an upper bound  on $p$ of $A q^{L} \cdot n \log(2n)$. Finally, we increase the value of $L$ in order to absorb the factor of $A$ into $q^L$.
\end{proof}

\begin{lem}\label{lem:bigbound} One can choose positive constants $A$ and $B$ so that the following holds. Let $n$ be a positive integer, and let $p_1, p_2, \dots, p_n$ be the first $n$ primes congruent to $1$ modulo $4$. One can choose $n+1$ distinct primes $q_1, q_2, \dots, q_{n+1}$ such that
\begin{enumerate}
\item for all $1 \le i \le n$, the prime $q_i$ is inert in $\Q(\sqrt{p_i})$ but split in $\Q(\sqrt{p_j})$ for $1 \le j\ne i \le n$,
\item $q_{n+1}$ is inert in all of the fields $\Q(\sqrt{p_1}), \dots, \Q(\sqrt{p_n})$,
\item each $q_i$ belongs to the interval $[2,A\cdot n^{B n}]$.
\end{enumerate}

\end{lem}
\begin{proof} 
%
We choose the $q_i$ inductively. Assume $q_j$ has already been chosen for $j < i$. The splitting condition on $q_i$ can be enforced by placing $q_i$ in a suitable residue class modulo $p_1 \cdots p_n$. Among the first $n+1$ primes in this residue class, there must be some choice of $q_i$ not equal to $q_j$ for any $j < i$. From Proposition \ref{prop:linnik},
\[ q_i \le (p_1 \cdots p_n)^{L} \cdot (n+1) \log(2n+2). \]
By the prime number theorem for progressions (or Proposition \ref{prop:linnik}), $p_n < Cn \log(2n)$ for some absolute constant $C$. Thus,
\[ (p_1 \cdots p_n)^{L} \cdot (n+1) \log(2n+2) \le p_n^{nL} \cdot (n+1)\log(2n+2) \le (Cn \log(2n))^{nL} \cdot (n+1)\log(2n+2). \]
It is straightforward to check that the final expression here is bounded by $A\cdot n^{Bn}$ for certain absolute positive constants $A$ and $B$.
\end{proof}

\begin{lem}\label{lem:woodlemma}
Let $q_1,\dots,q_{n+1}$ be distinct primes. The number of imaginary quadratic fields $k$ in which $q_1, \dots,q_n$ are inert, $q_{n+1}$ splits, and which satisfy $\abs{\Delta_k}<x$, is asymptotic to $cx$ as $x\to\infty$, where $c>0$. In fact, $c\ge \frac{3}{\pi^2} \cdot \frac{1}{3^{n+1}}$.
\end{lem}
\begin{proof} Let $P$ be a property of quadratic fields. We define the \emph{probability of $P$}, taken over the class of quadratic fields, as the limit
\[ \lim_{x\to\infty} \frac{\#\{\text{quadratic fields $k$ possessing $P$, having $|\Delta_{k}|\le x$}\}}{\#\{\text{quadratic fields $k$ with $|\Delta_{k}|\le x$}\}}, \]
provided that the limit exists. It is well-known that the denominator here is asymptotic to $\frac{6}{\pi^2}x$, as $x\to\infty$. Results of Wood \cite{wood10}, as collected in \cite[Proposition 3.3]{LMPT}, imply that
\begin{multline*} \Prob(\text{$k$ imaginary quadratic, $q_{n+1}$ splits, and $q_1, \dots, q_n$ inert}) \\ = \Prob(\text{$k$ imaginary quadratic}) \cdot \Prob(\text{$q_{n+1}$ splits}) \cdot \prod_{i=1}^{n} \Prob(\text{$q_i$ inert}),  \end{multline*}
that
\[ \Prob(\text{$k$ imaginary quadratic}) = \frac{1}{2}, \]
and that
\[ \Prob(\text{$q_{n+1}$ splits}) = \frac{1}{2}(1-\Prob(\text{$q_{n+1}$ ramifies})), \]
\[ \Prob(\text{$q_i$ inert}) = \frac{1}{2}(1-\Prob(\text{$q_i$ ramifies})) \quad(1\le i \le n). \]
Now for each prime $\ell$, the probability that $\ell$ ramifies is $\frac{1}{\ell+1}$. This last fact may be proved directly by obtaining an asymptotic formula for the count of fundamental discriminants divisible by $\ell$; this amounts to a problem on squarefree numbers in arithmetic progressions, whose solution is classical. In fact, this exact probability calculation is implicit in the proof of \cite[Theorem 1.5]{pollack}. Thus,
\[ \frac{1}{2}(1-\Prob(\text{$\ell$ ramifies})) = \frac{\ell}{2\ell+2} \ge \frac{1}{3}. \]
Piecing everything together, we find that
\[ \Prob(\text{$k$ imaginary quadratic, $q_{n+1}$ splits, $q_1, \dots, q_n$ inert}) \ge \frac{1}{2} \cdot(1/3)^{n+1}.  \]
The lemma follows.
\end{proof}
	
We now prove Theorem \ref{theorem:theorem2}.

Let $L_1,\dots, L_n$ be real quadratic fields of prime discriminants $p_1,\dots,p_n$, where $p_i\equiv 1 \pmod{4}$. For every $i\leq n$, let $q_i$ be a prime which is inert in $L_i/\Q$ and which splits in $L_j/\Q$ for any $i\neq j$. Let $q$ be a prime which is inert in all of the extensions $L_i/\Q$. We assume that the $n+1$ primes $q, q_1,\dots, q_n$ are distinct. Finally, let $B_1,\dots,B_n$ be quaternion algebras over $\Q$ such that $\Ram(B_i)=\{q, q_i\}$. The Albert-Brauer-Hasse-Noether theorem implies that if $k$ is a number field, $L/k$ a quadratic extension of fields and $B$ a quaternion algebra defined over $k$, then $B$ admits an embedding of $L$ if and only if no prime of $k$ which ramified in $B$ splits in $L/k$. In particular this shows that each quaternion algebra $B_i$ admits an embedding of the field $L_i$ and does not admit an embedding of $L_j$ for any $1\leq j\neq i\leq n$.

For each $i\leq n$, let $\mathcal O_{B_i}$ be a maximal order in $B_i$ and $\Gamma_{\mathcal O_{B_i}}^1$ the associated arithmetic Fuchsian group. Theorem 12.2.6 of \cite[Chapter 12]{MR}, which is stated for arithmetic Kleinian groups but holds {\em mutatis mutandis} for arithmetic Fuchsian groups, shows that we may select a hyperbolic element $\gamma_i\in \Gamma_{\mathcal O_{B_i}}^1$ whose eigenvalue $\lambda(\gamma_i)$ of largest absolute value satisfies $L_i=\Q(\lambda(\gamma_i))$. Note that the Fuchsian groups $\{ \Gamma_{\mathcal O_{B_i}}^1 \}$ are pairwise noncommensurable as they are derived from quaternion algebras which are not isomorphic over $\Q$. Borel's volume formula \cite{borel-commensurability} (see also \cite[Chapter 11.1]{MR}) and Lemma \ref{lem:bigbound} show that

\[ \mathrm{coarea}(\Gamma_{\mathcal O_{B_i}}^1) \leq \frac{\pi}{3}(q-1)(q_i-1) \leq c_1 n^{c_2n}\]

for absolute constants $c_1,c_2>0$. We may bound $\ell(\gamma_i)$ using the same argument employed in the proof of Theorem \ref{theorem:theorem1}. This argument, along with the bound $p_i<Cn\log(2n)$ for some absolute constant $C>0$ which follows from the prime number theorem for arithmetic progressions, shows that we may take $\ell(\gamma_i) < c_3 \left(n\log(2n)\right)^2$ for some positive constant $c_3$.

Suppose now that $k$ is an imaginary quadratic field in which $q$ splits and in which all of the $q_i$ are inert. This implies that each of the $k$-quaternion algebras $B_1\otimes_{\Q} k,\dots, B_n\otimes_{\Q} k$ are division algebras. In fact, we may conclude from Theorem \ref{theorem:tgs} that these algebras are all isomorphic to the same quaternion division algebra $B$ over $k$. Note that $B$ is the algebra over $k$ for which $\disc(B)=q\mathcal O_k$. Conjugating the orders $\mathcal{O}_{B_1},\dots,\mathcal{O}_{B_n}$ if necessary, we may assume that the maximal orders $\mathcal{O}_{B_1}\otimes_{\textbf{Z}} \mathcal{O}_k,\dots,\mathcal{O}_{B_n}\otimes_{\textbf{Z}} \mathcal{O}_k$ are all equal. Here we have used the fact that all maximal orders in an indefinite quaternion algebra over $\Q$ are conjugate. Denote by $\mathcal O$ the resulting maximal order of $B$. This discussion shows that the Kleinian group $\Gamma_{\mathcal O}^1$ contains all of the Fuchsian groups $\Gamma_{\mathcal O_{B_i}}^1$ and hence $\textbf{H}^3/\Gamma_{\mathcal O}^1$ contains $n$ geodesics lying on pairwise noncommensurable totally geodesic surfaces, all of whose lengths and areas have been bounded above. Using the trivial estimate $\zeta_k(2)<\zeta(2)^2$ and Lemma \ref{lem:bigbound}, we conclude from Borel's volume formula \cite{borel-commensurability} that there are constants $c_4,c_5>0$ such that

\[ \mathrm{covol}(\Gamma_{\mathcal O}^1)=\frac{\abs{\Delta_k}^{\frac{3}{2}}\zeta_k(2)}{4\pi^2}\prod_{\mathfrak q\mid q\mathcal O_k}(\abs{\mathfrak q}-1)\leq  c_4n^{c_5n}\abs{\Delta_k}^{\frac{3}{2}}.\]

Theorem \ref{theorem:theorem2} now follows from Lemma \ref{lem:woodlemma}.

\subsection*{Acknowledgments} The first author was partially supported by NSF RTG grant DMS-1045119 and an NSF Mathematical Sciences Postdoctoral Fellowship. The third author was partially supported by NSF grant DMS-1402268.

We thank \texttt{mathoverflow} user `GH from MO' for a post calling attention to the form of Linnik's theorem appearing as Proposition \ref{prop:linnik}.



\end{document}